\documentclass[fleqn,preprint,3p,a4paper]{elsarticle}
\usepackage{amsmath,amsthm,amssymb} 
\newtheorem{theorem}{Theorem}[section]

\newtheorem{proposition}[theorem]{Proposition}

\newtheorem{definition}[theorem]{Definition}
\newtheorem{remark}[theorem]{Remark}
\newtheorem{example}[theorem]{Example}
\newdefinition{problem}[theorem]{Problem}
\journal{Topology and its Applications}

\begin{document}
\begin{frontmatter}
\title{On some kinds of weakly sober spaces\tnoteref{t1}}
\tnotetext[t1]{This research is supported by  the National Natural Science Foundation of China (No. 11661057)and the Natural Science Foundation of Jiangxi Province (No. 20161BAB2061004).}
\author[mymainaddress]{Xinpeng Wen}
\ead{wenxinpeng2009@163.com}
\author[mysecondaryaddress]{Xiaoquan Xu\corref{mycorrespondingauthor}}
\cortext[mycorrespondingauthor]{Corresponding author}
\ead{xiqxu2002@163.com}
\address[mymainaddress]{College of Mathematics and Information, Nanchang Hangkong University, Nanchang, Jiangxi 330063, China}
\address[mysecondaryaddress]{School of Mathematics and Statistics, Minnan Normal University, Zhangzhou, Fujian 363000, China}

\begin{abstract}
In \cite{E_2018}, Ern\'e relaxed the concept of sobriety in order to extend the theory of sober spaces and locally hypercompact spaces to situations where directed joins were missing, and introduced three kinds of non-sober spaces: cut spaces, weakly sober spaces, and quasisober spaces. In this paper, their basic properties are investigated. It is shown that some properties which are similar to that of sober spaces hold and others do not hold.
\end{abstract}


\end{frontmatter}

\section{Introduction}
In \cite{E_2018}, Ern\'e relaxed the concept of sobriety in order to extend the theory of sober spaces and locally hypercompact spaces to situations where directed joins were missing. To that aim, he replaced joins by cuts, and introduced three kinds of non-sober spaces: cut spaces, weakly sober spaces, and quasisober spaces. This approach generalized and facilitated many results in the theory of quasicontinuous posets, and results about the more
restricted quasialgebraic domains and $s_2$-quasialgebraic posets are then easy consequences. For example, Ern\'e \cite{E_2018} proved that the locally hypercompact (resp., hypercompactly based) weakly sober spaces (quasisober spaces, cut spaces) are exactly the weak Scott spaces
of $s_2$-quasicontinuous (resp., $s_2$-quasialgebraic) posets in the sense of Zhang and Xu \cite{ZX_2015}. For weakly sober (cut, quasisober) $C$-spaces (resp., $B$-spaces), we have the similar results (see \cite{E_2018}).

In this paper, we will investigate the basic properties of cut spaces, weakly sober spaces and quasisober spaces. It is shown that some properties which are similar to that of sober spaces hold and others do not hold. The main results are: 

(1) We illustrate that a closed subspace of a quasisober space (cut space, weakly sober space) is not always a quasisober space (cut space, weakly sober space) by presenting a counterexample.

(2) A saturated subspace of a weakly sober space (cut space) is a weakly sober space (cut space).
(3) We show that a saturated subspace of a quasisober space is not always a quasisober space by presenting a counterexample.

(4) An open subspace of a quasisober space is quasisober.

(5) By presenting a counterexample we illustrate that the products of quasisober spaces (cut spaces, weakly sober spaces) is not always a quasisober space (cut space, weakly sober space).

(6) We show that a retract of a quasisober space (cut space, weakly sober space) is not always a quasisober space (cut space, weakly sober space) by presenting a counterexample.

(7) By presenting a counterexample we illustrate that a image subspace of continuous closure operator of a quasisober space (cut space, weakly sober space) is not always a quasisober space (cut space, weakly sober space).

(8) A image subspace of a continuous kernel operator of a quasisober space (cut space, weakly sober space) is a quasisober space(cut space, weakly sober space).

(9) If two topological spaces $X$ and $Y$ are quasisober spaces (cut spaces, weakly sober spaces), then the set $Top(X,Y)$ of all continuous functions $f:X\rightarrow Y$ equipped with the topology of pointwise convergence is not always a quasisober space (cut space, weakly sober space).

(10) We show that the Smyth power space of a quasisober space (cut space, weakly sober space) is not always a quasisober space (cut space, weakly sober space) by presenting a counterexample.

(11) If the Smyth power space  ~$P^s(X)$ of a topological space $X$ is a cut space, then $X$ is a cut space.

(12) If the Smyth power space $P^s(X)$ of a $T_0$ topological space $X$ is weakly sober, then $X$ is weakly sober.

(13) By presenting a counterexample we illustrate that the quasisoberity of the Smyth power space $P^s(X)$ of a topological space $X$ does not always imply the quasisoberity of ~$X$.

(14) If the Smyth power space $P^s(X)$ of a $T_0$ well-filtered topological space $X$ is quasisober, then $X$ is quasisober.

(15) If the Smyth power space $P^s(X)$ of a $T_0$ topological space $X$ is quasisober and the order of specialization on $X$ is a sup semilattice, then $X$ is quasisober.

\section{Preliminaries}

First we recall some basic definitions and notions used in this note; more details can be found in \cite{GHKLMS_2003, JGL_2013, RE_1989}. For a poset $P$ and $A\subseteq B\subseteq P$, let $\uparrow\! A=\{x\in P: a\leq x$ for some $a\in A\}$ (dually $\downarrow\! A=\{x\in P:x\leq a$ for some $a\in A\}$).
$A^{\uparrow}$ and $A^{\downarrow}$ denote the sets of all upper and lower bounds of $A$ in $P$, respectively. $A^{\uparrow_B}$ and $A^{\downarrow_B}$ denote the sets of all upper and lower bounds of $A$ in $B$, respectively. The set of all ideals in $P$ is denoted by $Id(P)$. Let $A^{\delta}=(A^{\uparrow})^{\downarrow}$, $A^{\delta_B}=(A^{\uparrow_B})^{\downarrow_B}$ and $\delta(P)=\{A^{\delta}: A\subseteq P\}$. $\delta(P)$ is called the $Dedekind$-$Macneille$ completion of $P$. $A^{\delta}$ is called the $cut$ $closure$ of $A$ in $P$. If $A^{\delta}=A$, we say that $A$ is a cut in $P$. A subset $A$ of $P$ is said to be an upper set if $A=\uparrow\! A$.  The Alexandroff topology $(P,up(P))$ on $P$ is the topology consisting of all its upper subsets. The family of all finite sets in $P$ is denoted by $P^{(<\omega)}$. $P$ is said to be a $directed$ $complete$ $poset$, a $dcpo$ for short, if every directed subset of $P$ has the least upper bound in $P$. The topology generated by the collection of sets $P\backslash\uparrow\! x$ (as subbasic open subsets) is called the $lower$ $topology$ on $P$ and denoted by $\omega(P)$; dually define the $upper$ $topology$ on $P$ and denote it by $\upsilon (P)$. The topology $\sigma(P)=\{U\subseteq P: U=\uparrow\! U$ and $U\cap D\neq\emptyset$ for each directed set $D$ with $\vee D\in U\}$ is called the $Scott$ $topology$. A $projection$ $operator$ is an idempotent, monotone self-map $f:P\rightarrow P$. A $closure$ $operator$ is a projection $c$ on $P$ with $1_P\leqslant c$. A $kernel$ $operator$ is a projection $k$ on $P$ with $k\leqslant 1_P$.

Given a topological space $(X,\tau)$, we can define a preorder $\leq_\tau$, called the $preorder$ $of$ $specialization$, by $x\leq_\tau y$ if and only if $x\in \mbox {cl}_\tau\{y\}$. Clearly, each open set is an upper set and each closed set is a lower set with respect to the preorder $\leq_\tau$. It is easy to see that $\leq_\tau$ is a partial order if and only if $(X,\tau)$ is a $T_0$ space. For any set $A\subseteq B \subseteq X$ we denote  $\uparrow_X\! A=\{x\in X: a\leq_\tau x$ for some $a\in A\}$, $\downarrow_X\! A=\{x\in X:x\leq_\tau a$ for some $a\in A\}$, $A^{\uparrow_X}=\{x\in X:a\leq_\tau x$ for all $a\in A\}$, $A^{\downarrow_X}=\{x\in X:x\leq_\tau a$ for all $a\in A\}$, $A^{\uparrow_B}=\{x\in B:a\leq_\tau x$ for all $a\in A\}$, $A^{\downarrow_B}=\{x\in B:x\leq_\tau a$ for all $a\in A\}$, $A^{\delta_X}=(A^{\uparrow_X})^{\downarrow_X}$ and $A^{\delta_B}=(A^{\uparrow_B})^{\downarrow_B}$. For $A=\{x\}$, $\uparrow_X\! A$ and $\downarrow_X\! A$ are shortly denoted by $\uparrow_X\! x$ and $\downarrow_X\! x$ respectively. A set is said to be $saturated$ in $(X,\tau)$ if it is the intersection of open sets, or equivalently if $A$ is an upper set (that is, $A=\uparrow_X\! A)$. $(X, \tau)$ is said to be $sober$ if it is $T_0$ and every irreducible closed set is a closure of a (unique) singleton. $Irr(X)$ denotes the set of all irreducible closed subsets of $(X, \tau)$. We denote by $K(X)$ the poset of nonempty compact saturated subsets of $(X, \tau)$ with the order reverse to containment, i.e., $K_1\leqslant K_2$ iff $K_2\subseteq K_1$. We consider the upper vietoris topology on $K(X)$, generated by the sets $\Box U=\{K\in K(X):K\subseteq U\}$, where $U$ ranges over the open subsets of $(X, \tau)$. The resulting topological space is called $the$ $Smyth$ $power$ $space$ or $the$ $upper$ $power$ $space$ $of$ $(X, \tau)$ and denoted by $P^s(X)$. For any closed subset $A$ of $(X, \tau)$, let $\diamond A=\{Q\in K(X):Q\cap A\neq\emptyset\}$. $(X, \tau)$ is called $well$-$filtered$, $WF$ for short, if for each filter basis $\mathcal{C}$ of compact saturated sets and open set $U$ with $\cap \mathcal{C}\subseteq U$, there is a $K\in \mathcal{C}$ with $K\subseteq U$. A subset $K$ of $(X, \tau)$ is $supercompact$ if for arbitrary families $(U_i)_{i\in I}$ of opens, $K\subseteq \bigcup\limits_{i\in I}U_i$ implies $K\subseteq U_k$ for some $k$ in $I$. Let $Y$ be a subset of $X$. The subspace topology on $Y$ has as opens the subsets of $Y$ of the form $U\cap Y$, $U$ open in $(X, \tau)$. We denote it by $(Y, \tau|Y)$. Let $X$ and $Y$ be two topological spaces. If $f:X\rightarrow Y$ and $g:Y\rightarrow X$ are continuous with $f\circ g=1_Y$, then $Y$ is called a $retract$ of $X$. The map $f$ is a $retraction$. Denote the set of all continuous maps from $X$ into $Y$ by $Top(X,Y)$. Give a point $x\in X$ and an open set $U$ of $Y$, let $S(x, U)=\{f\in Top(X,Y):f(x)\in U\}$, the sets $S(x, U)$ are a subbasis for topology on $Top(X,Y)$, which is called the topology of $pointwise$ $convergence$.
Unless otherwise stated, throughout the paper, whenever an order-thoretic concept is mentioned in the context of a topological space $(X, \tau)$, it is to be interpreted with respect to the specialization preorder on $(X, \tau)$.

Let $P$ be a poset. In the absence of enough (directed) joins, the weak Scott topology $\sigma_2(P)$ is often a good substitute for the classical Scott topology $\sigma(P)$. Note that a set is $\sigma_2$-open if and only if $D\cap U=\emptyset$ implies $D^{\delta}\cap U=\emptyset$ for all directed sets $D$. A space is said to be montone determined iff any subset $U$ is open whenever any montone net converging to a point in $U$ is eventually in $U$ (see \cite{E_2009}). By a cut space, we mean a space in which any montone net converges to each point in the cut closure of its range, or equivalently, the topological closure of any directed subset $D$ coincides with its cut closure $D^{\delta}$.

\begin{remark}[\cite{E_2018}]\label{remark:2.1}
$(1)$ The weak Scott topology  is the coarsest $($weakest$)$ topology on a qoset $P$ making it a montone determined space with specialization qoset $P$;

$(2)$ The cut spaces are exactly those space whose topology is coarser than the weak Scott topology of the specialization qoset;

$(3)$ The montone determined cut spaces are exactly the weak Scott spaces.
\end{remark}

\begin{definition}[\cite{E_2018}]\label{definition:2.2} \rm Let $X$ be a $T_0$ space.
\begin{enumerate}
\item[{(1)}]~$X$ is called $weakly$ $sober$ if every irreducible closed set is a cut, that is, an intersection of point closures (with the whole space as intersection of the empty set);

\item[{(2)}]~$X$ is called $quasisober$ if each irreducible closed set is the cut closure of a directed set.
\end{enumerate}
\end{definition}

\begin{remark}[\cite{E_2018}]\label{remark:2.3}
$(1)$ In a cut space, the cut closure of directed sets is irreducible and closed;

$(2)$ Every sober space is quasisober, every quasisober space is weakly sober, and every weakly sober space is a cut space;

$(3)$ A $T_0$ space is sober iff it is quasisober and directed complete in its specialization order;

\end{remark}

But a weakly sober space which is a dcpo in its specialization order is not always sober (see example 2.6). Indeed, a $T_0$ space is sober iff it is weakly sober and every irreducible (closed) subset has a sup in its specialization order.

\begin{example} [\cite{E_2018}]\label{example:2.4} Adding top $\top$ and bottom $\bot$ to an infinite antichain yields a noetherian lattice $L$ that is a $d$-space, hence a cut space, but not weakly sober in the topology $\upsilon(L)\cup \{\{\top\}\}$, because $L\backslash \{\top\}$ is irreducible and closed but not a cut.
\end{example}

\begin{example} [\cite{E_2018}]\label{example:2.5} The Scott space $\Sigma R$ of the real line $R$ is not sober $($as $R$ is not a dcpo$)$ but quasisober: the irreducible closed sets are the point closures and the whole space.
\end{example}

\begin{example} [\cite{E_2018}]\label{example:2.6} An infinite space $X$ with the cofinite topology is weakly sober but not quasisober: the irreducible closed sets are $X$ and the singletons. Obviously, $X$ isn't sober.
\end{example}

\section{Main results}

By the following example, we know that a closed subspace of a quasisober space (cut space, weakly sober space) is not always a quasisober space (cut space, weakly sober space).

\begin{example}\label{example:3.1} Let $L$=$\{a_1, a_2\}\cup N$ and $L^{'}$=$\{a_2\}\cup N$, where $N$ is the set of all natural numbers $\{1, 2, 3,\cdots, n,$\\$ \cdots\}$, be two posets with the partial order defined by $\forall n\in N, n<a_1, n<a_2, n<n+1$. But $a_1$ and $a_2$ are incomparable. We consider the two Alexandroff topological spaces $(L, up(L))$ and $(L^{'}, up(L^{'}))$. Then $L^{'}$ is a closed set in $(L, up(L))$ and we have that $(L^{'}, up(L^{'}))=(L^{'}, up(L)|L^{'})$. Because $Irr((L, up(L)))=Id(L)=\{\downarrow a: a\in L\}\cup \{N\}$ and $N= N^{\delta_{L}}$ where $N^{\delta_{L}}$ is the cut closure of $N$ in $L$, $(L, up(L))$ is quasisober. Since $N$ is a directed closed set with $N\neq N^{\delta_{L^{'}}}=L^{'}$ where $N^{\delta_{L^{'}}}$ is the cut closure of $N$ in $L^{'}$, $(L^{'}, up(L)|L^{'})$ isn't a cut space.
\end{example}

\begin{proposition}\label{proposition:3.2} A saturated subspace of a cut space is a cut space.
\end{proposition}
\begin{proof} Let $(X, \mathcal{O}(X))$ be a cut space, $U$ a saturated subset of $X$ and $D$ a directed subset of $U$. As $D^{\uparrow_U}=D^{\uparrow}$, we have that $D^{\delta}\cap U=D^{\delta_U}$ and thus $D^{\delta_U}=D^{\delta}\cap U=\mbox{cl}_XD\cap U=\mbox{cl}_UD$. Therefore, $(U, \mathcal{O}(X)|U)$ is a cut space.
\end{proof}

\begin{proposition}\label{proposition:3.3} A saturated subspace of a weakly sober space is weakly sober.
\end{proposition}
\begin{proof} Let $(X, \mathcal{O}(X))$ be a weakly sober space, $U$ a saturated subset of $(X, \mathcal{O}(X))$ and $A$ an irreducible closed subset of $(U, \mathcal{O}(X)|U)$. Then it is easy to see that $(U, \mathcal{O}(X)|U)$ is a $T_0$ space. Now we prove  that $\mbox{cl}_XA$ is an irreducible closed subset in $(X, \mathcal{O}(X))$. Let $U_1, U_2\in \mathcal{O}(X)$, $U_1\cap \mbox{cl}_XA\neq\emptyset$ and $U_2\cap \mbox{cl}_XA\neq\emptyset$. Then $U_1\cap A=U_1\cap A\cap U\neq\emptyset$ and $U_2\cap A=U_2\cap A\cap U\neq\emptyset$. Since $A$ is an irreducible closed subset in $(U, \mathcal{O}(X)|U)$, we have $U_1\cap U_2\cap U\cap A=U_1\cap U_2\cap A\neq\emptyset$ and thus $U_1\cap U_2\cap\mbox{cl}_XA\neq\emptyset$. Hence, $\mbox{cl}_XA$ is an irreducible closed subset in $(X, \mathcal{O}(X))$. As $(X, \mathcal{O}(X))$ is a weakly sober space, we have $\mbox{cl}_XA=(\mbox{cl}_XA)^{\delta}$ and thus $A=\mbox{cl}_XA\cap U=(\mbox{cl}_XA)^{\delta}\cap U$. Now we show that $A=A^{\delta_U}$. It is easy to see that $A\subseteq A^{\delta_U}$. Conversely, as $A^{\uparrow_U}=A^{\uparrow}$, we have $A^{\delta_U}=A^{\delta}\cap U$ and thus $A^{\delta_U}=A^{\delta}\cap U\subseteq (\mbox{cl}_XA)^{\delta}\cap U=\mbox{cl}_XA\cap U=A$. So $A=A^{\delta_U}$. Therefore, $(U, \mathcal{O}(X)|U)$ is a weakly sober space.
\end{proof}

By the following example, we know that a saturated subspace of a quasisober space is not always quasisober.

\begin{example}\label{example:3.4} Let $L$=$\{b_i: i\in N\}\cup N$, where~$N$ denotes the set of all natural numbers. Define a partial order on $L$ by setting: ~$\forall n\in N$,~$\downarrow n=\{1, 2, 3,\cdots, n\}$ ~$\downarrow b_n=\{b_n\}$. Consider the topological space $(L, \upsilon(L))$. Then we have $Irr((L, \upsilon(L)))=\{\downarrow\! a:a\in L\}\cup\{L\}$ and $L=N^{\delta_L}=\mbox{cl}_{(L, \upsilon(L))}N$. Thus $(L, \upsilon(L))$ is quasisober. Since $L\backslash N=\{b_i: i\in N\}$ is a saturated subset in $(L, \upsilon(L))$ and the topological space $(L\backslash N, \upsilon(L)|L\backslash N)$ is exactly the topological space $(L\backslash N, \upsilon(L\backslash N))$, we have $Irr((L\backslash N, \upsilon(L)|L\backslash N))=\{\{b_i\}:i\in N\}\cup \{L\backslash N\}$. Therefore, $(L\backslash N, \upsilon(L)|L\backslash N)$ isn't quasisober.
\end{example}

\begin{proposition}\label{proposition:3.5} An open subspace of a quasisober space is quasisober.
\end{proposition}
\begin{proof}
Let $(X, \mathcal{O}(X))$ be a quasisober space, $U$ an open subset of $(X, \mathcal{O}(X))$ and $A$ an irreducible closed subset of $(U, \mathcal{O}(X)|U)$. Trivially, $(U, \mathcal{O}(X)|U)$ is a $T_0$ space and $\mbox{cl}_XA$ is an irreducible closed subset in $(X, \mathcal{O}(X))$. It is easy to see that $A=\mbox{cl}_XA\cap U$. Since every quasisober space is weakly sober and a saturated subspace of a weakly sober space is weakly sober, we have $A^{\delta_U}=A=\mbox{cl}_XA\cap U$. Since $(X, \mathcal{O}(X))$ is quasisober, there exists a directed  subset $D\subseteq X$ such that $\mbox{cl}_XA=D^{\delta}=\mbox{cl}_XD$. Thus, $A=A^{\delta_U}=\mbox{cl}_XA\cap U=\mbox{cl}_XD\cap U=D^{\delta}\cap U\neq\emptyset$. So there exists $d\in D\cap U$. Let $E=\uparrow\! d\cap D$. Then $E$ is a directed subset in $U$ and we have $E^{\delta}=D^{\delta}=\mbox{cl}_XE$. Thus, we have $A=D^{\delta}\cap U=E^{\delta}\cap U=E^{\delta_U}$. Therefore, $(U, \mathcal{O}(X)|U)$ is quasisober.
\end{proof}

By the following example, we know that products of quasisober spaces (cut spaces, weakly sober spaces) are not always quasisober spaces (cut spaces, weakly sober spaces).

\begin{example}\label{example:3.6} Let $N$ be the set of all natural numbers with their usual ordering. Consider the Alexandroff topological space $(N, up(N))$. Then we have $Irr((N, up(N)))=Id(N)=\{\downarrow n: n\in N\}\cup \{N\}$ and $N=N^{\delta}$. Thus, $(N, up(N))$ is quasisober. Let $D=N\times\{1\}$. Then $D$ is a directed closed subset in $(N\times N, up(N)\times up(N))$. Thus, $D$ is an irreducible closed subset in $(N\times N, up(N)\times up(N))$ and we have $D^{\delta_{N\times N}}=N\times N\neq D$. Therefore, $(N\times N, up(N)\times up(N))$ isn't a cut space.
\end{example}

The next example shows that a retract of a quasisober space (cut space, weakly sober space) may not be a quasisober space (cut space, weakly sober space).

\begin{example}\label{example:3.7} Let $L$ and $L^{'}$ be the two posets in example $\ref{example:3.1}$. Consider the two topological spaces $(L, up(L))$ and $(L^{'}, up(L^{'}))$. Clearly, $(L, up(L))$ is quasisober, while $(L^{'}, up(L^{'}))$ is not a cut space. The mapping $f:(L, up(L))\longrightarrow (L^{'}, up(L^{'}))$ defined by the formula
\begin{equation}
f(x)=
\begin{cases}
a_2 & x\in \{a_1, a_2\}\\
x & x\in N
\end{cases}
\end{equation}
is continuous. The mapping ~$j:(L^{'}, up(L^{'}))\longrightarrow (L, up(L))$ defined by $\forall x\in L^{'}, j(x)=x$ is continuous. Then for any $x\in L^{'}, f\circ j(x)=f(x)=x$, that is, $(L^{'}, up(L^{'}))$ is a retract of $(L, up(L))$.
\end{example}

Next, we will give an example to show that a image subspace of continuous closure operator of a quasisober space (cut space, weakly sober space) may not be a quasisober space (cut space, weakly sober space).

\begin{example}\label{example:3.8}
Let $L=\{\top, a_1, a_2\}\cup N$, where $N$ denotes the set of all natural numbers. Define a partial order on $L$ by setting:\\
~$\downarrow \top=L$,\\
~$\downarrow a_1=\{a_1\}\cup N$,\\
~$\downarrow a_2=\{a_2\}\cup N$,\\
~$\downarrow n=\{1, 2, 3,\cdots, n\}$ ~$(\forall n\in N)$.\\
Consider the topological space $(L, up(L))$. Define the mapping $p:(L, up(L))\longrightarrow (L, up(L))$ by the formula
\begin{equation}
p(x)=
\begin{cases}
\top & x\in \{\top, a_1, a_2\}\\
x & x\in N
\end{cases}
\end{equation}. Clearly, $p$ is continuous. It is easy to see that $p$ is a closure operator and we have $p(L)=\{\top\}\cup N$. Let $L^{'}=p(L)=\{\top\}\cup N$. Then we have $(L^{'}, up(L)|L^{'})=(L^{'}, up(L^{'}))$. Since $Irr((L, up(L)))=Id(L)=\{\downarrow a:a\in L\}\cup \{N\}$ and $N=N^{\delta}$, $(L, up(L))$ is quasisober. Since $N$ is an irreducible closed subset in $(L^{'}, up(L^{'}))$ and since $\mbox{cl}_{(L^{'}, up(L^{'}))}N=N\neq N^{\delta_{L^{'}}}=N\cup \{\top\}$, $(L^{'}, up(L)|L^{'})$ is not a cut space.
\end{example}

\begin{proposition}\label{proposition:3.9} If $(X,\mathcal{O}(X))$ is a cut space and a continuous mapping $p:(X,\mathcal{O}(X))\longrightarrow (X,\mathcal{O}(X))$ is a kernel operator with respect to the specialization preorder of $(X,\mathcal{O}(X))$, then the subspace $(p(X), \mathcal{O}(X)|p(X))$ is a cut space.
\end{proposition}
\begin{proof}
Suppose that $D\subseteq p(X)$ is a directed subset in $(X,\mathcal{O}(X))$. Now we claim that $\uparrow\! (D^{\uparrow_{p(X)}})=D^{\uparrow}$. Clearly, $\uparrow\! (D^{\uparrow_{p(X)}})\subseteq D^{\uparrow}$. Conversely, since continuous mapping $p$ is a kernel operator, for any $x\in D^{\uparrow}$ we have $p(x)\leqslant x$ and $p(x)\in D^{\uparrow_{p(X)}}$. Thus $x\in \uparrow\!(D^{\uparrow_{p(X)}})$. Hence, $\uparrow\!(D^{\uparrow_{p(X)}})\supseteq D^{\uparrow}$. So $\uparrow\!(D^{\uparrow_{p(X)}})=D^{\uparrow}$. Thus, $D^{\delta_{p(X)}}=(\uparrow\!(D^{\uparrow_{p(X)}}))^{\downarrow_{p(X)}}=D^{\delta}\cap p(X)$. As $(X,\mathcal{O}(X))$ is a cut space, we have $D^{\delta_{p(X)}}=D^{\delta}\cap p(X)=\mbox{cl}_XD\cap p(X)=\mbox{cl}_{p(X)}D$. Therefore, $(p(X), \mathcal{O}(X)|p(X))$ is a cut space.
\end{proof}

\begin{proposition}\label{proposition:3.10} If $(X,\mathcal{O}(X))$ is a weakly sober space and a continuous mapping $p:(X,\mathcal{O}(X))\longrightarrow (X,\mathcal{O}(X))$ is a kernel operator with respect to the specialization preorder of $(X,\mathcal{O}(X))$, then the subspace $(p(X), \mathcal{O}(X)|p(X))$ is a weakly sober space.
\end{proposition}
\begin{proof}
Clearly, $(p(X), \mathcal{O}(X)|p(X))$ is a $T_0$ space. Suppose that $A\subseteq p(X)$ is an irreducible closed subset in $(p(X), \mathcal{O}(X)|p(X))$. Then $\mbox{cl}_XA$ is an irreducible closed subset in $(X,\mathcal{O}(X))$. As $(X,\mathcal{O}(X))$ is a weakly sober space, we have $\mbox{cl}_XA=(\mbox{cl}_XA)^{\delta}$. Now we claim that $\uparrow\!(A^{\uparrow_{p(X)}})=A^{\uparrow}$. Clearly, $\uparrow\!(A^{\uparrow_{p(X)}})\subseteq A^{\uparrow}$. Conversely, since continuous mapping $p$ is a kernel operator, for any $x\in A^{\uparrow}$ we have $p(x)\leqslant x$ and $p(x)\in A^{\uparrow_{p(X)}}$. Thus $x\in \uparrow\!(A^{\uparrow_{p(X)}})$. Hence $\uparrow\!(A^{\uparrow_{p(X)}})\supseteq A^{\uparrow}$. So $\uparrow\!(A^{\uparrow_{p(X)}})=A^{\uparrow}$. Thus, $A^{\delta_{p(X)}}=(\uparrow\!(A^{\uparrow_{p(X)}}))^{\downarrow_{p(X)}}=A^{\delta}\cap p(X)$. As $A^{\delta_{p(X)}}\supseteq A=\mbox{cl}_XA\cap p(X)=(\mbox{cl}_XA)^{\delta}\cap p(X)\supseteq A^{\delta}\cap p(X)=A^{\delta_{p(X)}}$, we have $A^{\delta_{p(X)}}=A$. Therefore, $(p(X),\mathcal{O}(X)|p(X))$ is a weakly sober space.
\end{proof}

\begin{proposition}\label{proposition:3.11} If $(X,\mathcal{O}(X))$ is a quasisober space and a continuous mapping $p:(X,\mathcal{O}(X))\longrightarrow (X,\mathcal{O}(X))$ is a kernel operator with respect to the specialization preorder of $(X,\mathcal{O}(X))$, then the subspace $(p(X), \mathcal{O}(X)|p(X))$ is a quasisober space.
\end{proposition}
\begin{proof}
Clearly, $(p(X), \mathcal{O}(X)|p(X))$ is a $T_0$ space. Suppose that $A\subseteq p(X)$ is an irreducible closed subset in $(p(X), \mathcal{O}(X)|p(X))$. Then $\mbox{cl}_XA$ is an irreducible closed subset in $(X,\mathcal{O}(X))$.  As $(X,\mathcal{O}(X))$ is a quasisober space, there is a directed subset $D\subseteq X$ such that $\mbox{cl}_XA=\mbox{cl}_XD=D^{\delta}$. Now we show that $\uparrow\!((p(D))^{\uparrow_{p(X)}})=(p(D))^{\uparrow}$. Clearly, $\uparrow\!((p(D))^{\uparrow_{p(X)}})\subseteq (p(D))^{\uparrow}$. Conversely, since continuous mapping $p$ is a kernel operator, for any $x\in (p(D))^{\uparrow}$ we have $p(x)\leqslant x$ and $p(x)\in (p(D))^{\uparrow_{p(X)}}$. Thus $x\in \uparrow\!((p(D))^{\uparrow_{p(X)}})$. Hence, $\uparrow\!((p(D))^{\uparrow_{p(X)}})\supseteq (p(D))^{\uparrow}$. So $\uparrow\!((p(D))^{\uparrow_{p(X)}})=(p(D))^{\uparrow}$. Thus, $(p(D))^{\delta_{p(X)}}=(\uparrow\!((p(D))^{\uparrow_{p(X)}}))^{\downarrow_{p(X)}}=(p(D))^{\delta}\cap p(X)$. Now prove that $(p(D))^{\delta}\cap p(X)\subseteq D^{\delta}\cap p(X)$. As $p$ is a kernel operator, we have $p(D)\subseteq \downarrow\!D$ and thus $(p(D))^{\delta}\subseteq (\downarrow\!D)^{\delta}=D^{\delta}$. Hence, $(p(D))^{\delta}\cap p(X)\subseteq D^{\delta}\cap p(X)$. Now we show that $D^{\delta}\cap p(X)\subseteq p(D^{\delta})$. Since for any $x\in D^{\delta}\cap p(X)$ we have $x=p(x)\in p(D^{\delta})$, it is easy to see that $D^{\delta}\cap p(X)\subseteq p(D^{\delta})$. Now we show that $p(D^{\delta})\subseteq (p(D))^{\delta}\cap p(X)$. Since $p$ is continuous and $(X,\mathcal{O}(X))$ is quasisober, we have $p(D^{\delta})=p(\mbox{cl}_XD)\subseteq \mbox{cl}_Xp(D)\subseteq (p(D))^{\delta}$ and thus $p(D^{\delta})\subseteq (p(D))^{\delta}\cap p(X)$. Hence, $(p(D))^{\delta}\cap p(X)\subseteq D^{\delta}\cap p(X)\subseteq p(D^{\delta})\subseteq (p(D))^{\delta}\cap p(X)$. Therefore, $D^{\delta}\cap p(X)=(p(D))^{\delta}\cap p(X)$. Since $A=\mbox{cl}_XA\cap p(X)=\mbox{cl}_XD\cap p(X)=D^{\delta}\cap p(X)=(p(D))^{\delta}\cap p(X)=(p(D))^{\delta_{p(X)}}$, we have $A=(p(D))^{\delta_{p(X)}}$ and  $p(D)$ is a directed subset in $p(X)$. Therefore,  $(p(X), \mathcal{O}(X)|p(X))$ is quasisober.
\end{proof}

Next, we will give an example to show that even if $Y$ is a quasisober space (cut space, weakly space), then the set $Top(X, Y)$ of all continuous functions $f:X\rightarrow Y$ equipped with the topology of pointwise convergence may not be a quasisober space (cut space, weakly space).

\begin{example}\label{example:3.12}
Let $L$ be the poset in example $\ref{example:3.1}$. Clearly, $(L, up(L))$ is quasisober. We denote by $[L\rightarrow L]$ the set of all order preserving functions from $L$ into $L$. For any $i\in N$, define the mapping $f_i:(L, up(L))\longrightarrow (L, up(L))$ by the formula
\begin{equation}
f_i(x)=
\begin{cases}
a_2 & x\in L\backslash \{1\}\\
i & x=1
\end{cases}
\end{equation}.
Let $D=\{f_i:i\in N\}$. Then $D$ is a directed subset in $[L\rightarrow L]$ under the pointwise partial order. Now we prove that $D^{\delta_{[L\rightarrow L]}}=\downarrow_{[L\rightarrow L]}g$, where $g(x)=a_2$ for any $x\in L$. One readily sees that $g$ is an upper bound of $D$ in $[L\rightarrow L]$. Let $g^{'}\in [L\rightarrow L]$ be another upper bound of $D$ in $[L\rightarrow L]$, that is, $g^{'}\in D^{\uparrow_{[L\rightarrow L]}}$. Then we have $g^{'}(1)\in \bigcap\limits_{n\in N}\uparrow f_n(1)=\bigcap\limits_{n\in N}\uparrow n=\{a_1, a_2\}$. Since $g^{'}(a_2)\in \{f_i(a_2):i\in N\}^{\uparrow}=\{a_2\}^{\uparrow}=\{a_2\}$, it follows that $g^{'}(a_2)=a_2\geqslant g^{'}(1)$ and thus we have $g^{'}(1)=a_2$. Hence, for any $x\in L$ we have $g^{'}(x)=a_2$, that is, $g^{'}=g$. So $\bigvee_{[L\rightarrow L]}D=g$. Therefore, it follows that $D^{\delta_{[L\rightarrow L]}}=\downarrow_{[L\rightarrow L]}g$. Now we claim that $\downarrow_{[L\rightarrow L]}D$ is a closed subset in $[L\rightarrow L]$ with respect to the topology of pointwise convergence. Let $f\in [L\rightarrow L]$. We consider four cases.\\
Case $1$. $f(a_1)\in N$, ~$f(a_2)\in N$.\\
Clearly, it follows that $f\in \downarrow_{[L\rightarrow L]}D$.\\
Case $2$. $f(a_1)\in N$, ~$f(a_2)\notin N$.\\
If $f(a_2)=a_2$, it is easy to see that $f\in \downarrow_{[L\rightarrow L]}D$.\\
If $f(a_2)=a_1$, then we have $f\in S(a_2, \{a_1\})$ and $S(a_2, \{a_1\})\cap \downarrow_{[L\rightarrow L]}D=\emptyset$, where $S(a_2, \{a_1\})$ denotes the set $\{f\in [L\rightarrow L]:f(a_2)\in \{a_1\}\}$.\\
Case $3$. $f(a_1)\notin N$, ~$f(a_2)\in N$.\\
If $f(a_1)=a_2$, then $f\in \downarrow_{[L\rightarrow L]}D$.\\
If $f(a_1)=a_1$, then $f\in S(a_1, \{a_1\})$ and $S(a_1, \{a_1\})\cap \downarrow_{[L\rightarrow L]}D=\emptyset$.\\
Case $4$. $f(a_1)\notin N$, ~$f(a_2)\notin N$.\\
If $f(a_1)=a_1$ and $f(a_2)=a_1$, then $f\in S(a_1, \{a_1\})\cap S(a_2, \{a_1\})$ and $S(a_1, \{a_1\})\cap S(a_2, \{a_1\})\cap \downarrow_{[L\rightarrow L]}D=\emptyset$.\\
If $f(a_1)=a_2$, $f(a_2)=a_2$ and $f(1)=a_2$, then $f\in S(1, \{a_2\})$ and $S(1, \{a_2\})\cap \downarrow_{[L\rightarrow L]}D=\emptyset$.\\
If $f(a_1)=a_2$, $f(a_2)=a_2$ and $f(1)\in N$, then it follows that $f\in \downarrow_{[L\rightarrow L]}D$.\\
If $f(a_1)=a_1$ and $f(a_2)=a_2$, then $f\in S(a_1, \{a_1\})\cap S(a_2, \{a_2\})$ and $S(a_1, \{a_1\})\cap S(a_2, \{a_2\})\cap \downarrow_{[L\rightarrow L]}D=\emptyset$.\\
If $f(a_1)=a_2$ and $f(a_2)=a_1$, then $f\in S(a_1, \{a_2\})\cap S(a_2, \{a_1\})$ and $S(a_1, \{a_2\})\cap S(a_2, \{a_1\})\cap \downarrow_{[L\rightarrow L]}D=\emptyset$.

All these together show that $\downarrow_{[L\rightarrow L]}D$ is a closed subset in $[L\rightarrow L]$ with respect to the topology of pointwise convergence. As $D^{\delta_{[L\rightarrow L]}}=\downarrow_{[L\rightarrow L]}g\neq \downarrow_{[L\rightarrow L]}D$, it follows that the topological space of pointwise convergence of $[L\rightarrow L]$ is not a cut space.
\end{example}

By the following example, the Smyth power space $P^s(X)$ of a quasisober space (cut space, weakly sober space) may not be a quasisober space (cut space, weakly sober space).

\begin{example}\label{example:3.13}
Let $L$ be the poset in example $\ref{example:3.1}$. Clearly, $(L, up(L))$ is quasisober. All nonempty compact saturated subsets $K(L)$ of $(L, up(L))$ are exactly $\{\uparrow a:a\in L\}\cup \{\{a_1, a_2\}\}$. Now we prove that $P^s((L,  up(L)))=(K(L),  up(K(L)))$. Clearly, we have $P^s((L,  up(L)))\subseteq (K(L),  up(K(L)))$. Conversely, suppose that $Q\in K(L)$. Since $Q$ is an open subset in $(L, up(L))$, it follows that $\uparrow_{K(L)}Q=\Box Q\in \mathcal{O}(P^s((L,  up(L))))$ and thus $P^s((L,  up(L)))\supseteq (K(L),  up(K(L)))$. Hence, $P^s((L,  up(L)))=(K(L),  up(K(L)))$. Let $D=\{\uparrow a:a\in N\}$. It is easy to see that $D=\{\uparrow a:a\in N\}$ is a directed lower subset of $K(L)$. Then $D$ is a directed closed  subset of $P^s((L,  up(L)))$. As $D^{\delta_{K(L)}}=D\cup \{\{a_1, a_2\}\}$, we have $D=\mbox{cl}_{P^s((L,  up(L)))}D\neq D^{\delta_{K(L)}}=D\cup \{\{a_1, a_2\}\}$. Therefore, $P^s((L,  up(L)))$ is not a cut space.
\end{example}

\begin{theorem}\label{theorem:3.14}
Let $(X, \mathcal{O}(X))$ be a topological space, and suppose that $P^s((X, \mathcal{O}(X)))$ is a cut space. Then $(X, \mathcal{O}(X))$ is a cut space.
\end{theorem}
\begin{proof}
Assume that $D$ is a directed subset of $X$ with respect to the specialization preorder of $(X, \mathcal{O}(X))$. Then $\{\uparrow d: d\in D\}$ is a directed subset of $K(X)$. Now we claim that $\mbox{cl}_{X}D=D^{\delta}$. Clearly, $\mbox{cl}_{X}D\subseteq D^{\delta}$. Conversely, let $x\in D^{\delta}$. Then $\bigcap\limits_{d\in D}\uparrow\! d\subseteq\uparrow\! x $. Thus for each $Q\in \{\uparrow d: d\in D\}^{\uparrow_{K(X)}}$ we have $Q\subseteq \bigcap\limits_{d\in D}\uparrow\! d\subseteq\uparrow\! x $. So $\uparrow\! x\in \{\uparrow d: d\in D\}^{\delta_{K(X)}}$. Assume that $U\in \mathcal{O}(X)$  and $x\in U$. As $P^s((X, \mathcal{O}(X)))$ is a cut space, it follows that $\uparrow\! x\in \Box U\cap \{\uparrow d: d\in D\}^{\delta_{K(X)}}=\Box U\cap \mbox{cl}_{P^s((X, \mathcal{O}(X)))}\{\uparrow d: d\in D\}$ and thus $\Box U\cap \{\uparrow d: d\in D\}\neq \emptyset$. So $U\cap D\neq \emptyset$. Thus $x\in \mbox{cl}_{X}D$. Hence $\mbox{cl}_{X}D\supseteq D^{\delta}$. So $\mbox{cl}_{X}D=D^{\delta}$. Therefore, $(X, \mathcal{O}(X))$ is a cut space.
\end{proof}

\begin{theorem}\label{theorem:3.15}
Let $(X, \mathcal{O}(X))$ be a $T_0$ space, and suppose that $P^s((X, \mathcal{O}(X)))$ is weakly sober. Then $(X, \mathcal{O}(X))$ is weakly sober.
\end{theorem}
\begin{proof}
Suppose that $A$ is an irreducible closed subset in $(X, \mathcal{O}(X))$. Then it follows that $\diamond A=\{Q\in K(X):Q\cap A\neq\emptyset\}=\mbox{cl}_{P^s((X, \mathcal{O}(X)))}\{\uparrow a: a\in A\}$ is an irreducible closed subset in $P^s((X, \mathcal{O}(X)))$. As $P^s((X, \mathcal{O}(X)))$ is weakly sober, we have $\diamond A=(\diamond A)^{\delta_{K(X)}}$. Now we show that $\bigcap\limits_{Q\in \diamond A}Q=\bigcap\limits_{a\in A}\uparrow a$. Clearly, $\bigcap\limits_{Q\in \diamond A}Q\subseteq\bigcap\limits_{a\in A}\uparrow a$. Conversely, let $y\in \bigcap\limits_{a\in A}\uparrow a$. For any $Q\in \diamond A$, there exists $b\in Q\cap A$ such that $b\leqslant y$; and hence $y\in Q$. Then $y\in \bigcap\limits_{Q\in \diamond A}Q$. Thus it follows that $\bigcap\limits_{Q\in \diamond A}Q\supseteq\bigcap\limits_{a\in A}\uparrow a$. Therefore, $\bigcap\limits_{Q\in \diamond A}Q=\bigcap\limits_{a\in A}\uparrow a$. Now we prove that $A=A^{\delta}$. Clearly, $A\subseteq A^{\delta}$. Conversely, let $x\in A^{\delta}$. Then we have $\bigcap\limits_{a\in A}\uparrow a\subseteq \uparrow x$. As $(\diamond A)^{\uparrow_{K(X)}}$$=\bigcap\limits_{Q\in \diamond A}\uparrow_{K(X)}Q$$=\{P\in K(X):P\in \bigcap\limits_{Q\in \diamond A}\uparrow_{K(X)}Q\}$$=\{P\in K(X):P\subseteq \bigcap\limits_{Q\in \diamond A}Q\}$, we have $(\diamond A)^{\delta_{K(X)}}$$=\{P\in K(X):\bigcap\limits_{Q\in \diamond A}Q\subseteq P\}$$=\{P\in K(X):\bigcap\limits_{a\in A}\uparrow a\subseteq P\}$; and hence $\uparrow x\in (\diamond A)^{\delta_{K(X)}}=\diamond A$. So $x\in A$. Thus $A\supseteq A^{\delta}$. Hence $A=A^{\delta}$. Therefore, $(X, \mathcal{O}(X))$ is weakly sober.
\end{proof}

Next, we will give an example to show that even if the Smyth power space $P^s(X)$ of a topological space $X$ is quasisober, $X$ may not be quasisober.

\begin{example}\label{example:3.16}
Let $N$ be an infinite antichain, where $N$ denotes the set of all natural numbers. Consider the upper topological space $(N, \upsilon(N))$. Clearly, $(N, \upsilon(N))$ is a $T_0$ space, the set $Id(N)$ is equal to the set $\{\{n\}:n\in N\}$ and $N$ is an irreducible closed subset of $(N, \upsilon(N))$. Thus, $(N, \upsilon(N))$ isn't quasisober and the set $K(N)$ of all nonempty compact saturated subsets of $(N, \upsilon(N))$ is equal to the set $2^{N}\backslash \{\emptyset\}$. As the mapping $j:(N, \upsilon(N))\longrightarrow P^s((N, \upsilon(N)))(x\longmapsto \{x\})$ is continuous, $\mbox{cl}_{P^s((N, \upsilon(N)))}j(N)=2^{N}\backslash \{\emptyset\}$ is an irreducible closed subset in $P^s((N, \upsilon(N)))$. Let $D=\{N\backslash \{1, 2, 3,\cdots, n\}:n\in N\}$. Clearly, $D$ is a directed subset of $K(N)$ with $\cap D=\emptyset$. Thus $D^{\delta_{K(N)}}=K(N)=2^{N}\backslash \{\emptyset\}$. Let $\mathcal{H}$ be an irreducible closed subset of $P^s((N, \upsilon(N)))$ with $\mathcal{H}\subsetneqq K(N)$. Then there is $\{B_i:i\in I\}\subseteq N^{(<\omega)}\backslash \{\emptyset\}$ such that $\bigcap\limits_{i\in I}\diamond B_i=\mathcal{H}$. Thus for any $i\in I$, there exists $b_{ij}\in B_i$ such that $\diamond \{b_{ij}\}\in \{\diamond B_i:i\in I\}$. Suppose not, that is, for any $b_{ij}\in B_i$ we have $N\backslash \{b_{ij}\}\in \bigcap\limits_{i\in I}\diamond B_i=\mathcal{H}$. Then $\Box (N\backslash \{b_{ij}\})\cap (\bigcap\limits_{i\in I}\diamond B_i)\neq \emptyset$. Since $\bigcap\limits_{i\in I}\diamond B_i$ is an irreducible closed subset of $P^s((N, \upsilon(N)))$ with $B_i\in N^{(<\omega)}\backslash \{\emptyset\}$, we have $\Box (N\backslash B_i)\cap(\bigcap\limits_{i\in I}\diamond B_i)=(\bigcap\limits_{b_{ij}\in B_i}\Box N\backslash \{b_{ij}\})\cap(\bigcap\limits_{i\in I}\diamond B_i)\neq\emptyset$, contradicting $\Box (N\backslash B_i)\cap\diamond B_i=\emptyset$. Let $G=\{b_{ij}:b_{ij}\in B_i$  and $\diamond\{b_{ij}\}\in \{\diamond B_i:i\in I\}\}$. Now we prove that $\bigcap\limits_{i\in I}\diamond B_i=\bigcap\limits_{a\in G}\diamond\{a\}$. Clearly, $\bigcap\limits_{i\in I}\diamond B_i\subseteq\bigcap\limits_{a\in G}\diamond\{a\}$. Conversely, for any $C\in \bigcap\limits_{a\in G}\diamond\{a\}$ and $i\in I$, there exists $b_{ij}\in B_i$ such that $\diamond \{b_{ij}\}\in \{\diamond B_i:i\in I\}$; and hence $C\cap B_i\neq\emptyset$. Thus $C\in \bigcap\limits_{i\in I}\diamond B_i$. So $\bigcap\limits_{i\in I}\diamond B_i\supseteq\bigcap\limits_{a\in G}\diamond\{a\}$. Therefore $\bigcap\limits_{i\in I}\diamond B_i=\bigcap\limits_{a\in G}\diamond\{a\}$. Now we prove that $\bigcap\limits_{a\in G}\diamond\{a\}=\downarrow_{K(N)}G$. Clearly, $\bigcap\limits_{a\in G}\diamond\{a\}\supseteq \downarrow_{K(N)}G$. Conversely, let $C\in \bigcap\limits_{a\in G}\diamond\{a\}$. Then for any $a\in G$ we have $a\in C$; and hence $G\subseteq C$. So $C\in \downarrow_{K(N)}G$. Thus $\bigcap\limits_{a\in G}\diamond\{a\}\subseteq \downarrow_{K(N)}G$. Hence $\bigcap\limits_{a\in G}\diamond\{a\}=\downarrow_{K(N)}G=\mathcal{H}$. All these together show that $P^s((N, \upsilon(N)))$ is quasisober, while $(N, \upsilon(N))$ isn't quasisober.
\end{example}

\begin{theorem}\label{theorem:3.17}
Let $(X, \mathcal{O}(X))$ be a $T_0$ space.
\begin{enumerate}
\item[\emph{(1)}] Suppose that $(X, \mathcal{O}(X))$ is a WF space. If $P^s((X, \mathcal{O}(X)))$ is quasisober, then $(X, \mathcal{O}(X))$ is quasisober;
\item[\emph{(2)}] Suppose that the order of specialization on $(X, \mathcal{O}(X))$ is a sup semilattice. If $P^s((X, \mathcal{O}(X)))$ is quasisober, then $(X, \mathcal{O}(X))$ is quasisober.
\end{enumerate}
\end{theorem}
\begin{proof}
Suppose that $A$ is an irreducible closed subset in $(X, \mathcal{O}(X))$. Then $\diamond A=\{Q\in K(X):Q\cap A\neq\emptyset\}=\mbox{cl}_{P^s((X, \mathcal{O}(X)))}\{\uparrow\! a: a\in A\}$ is an irreducible closed subset in $P^s((X, \mathcal{O}(X)))$. As $P^s((X, \mathcal{O}(X)))$ is quasisober, there exists a directed subset $\mathcal{F}\subseteq K(X)$ such that $\mathcal{F}^{\delta_{K(X)}}=\diamond A$ with $\mathcal{F}\subseteq \diamond A$. Let $\mathcal{H}=\{\uparrow\! (F\cap A):F\in \mathcal{F}\}$. Then $\mathcal{H}$ is directed subset of $K(X)$ with $\mathcal{H}\subseteq \diamond A$. Thus $\mathcal{H}^{\delta_{K(X)}}\subseteq \diamond A=\mathcal{F}^{\delta_{K(X)}}$. Conversely, since $\mathcal{H}^{\uparrow_{K(X)}}\subseteq \mathcal{F}^{\uparrow_{K(X)}}$, it follows that $\mathcal{H}^{\delta_{K(X)}}\supseteq \diamond A=\mathcal{F}^{\delta_{K(X)}}$. Hence $\mathcal{H}^{\delta_{K(X)}}=\diamond A=\mathcal{F}^{\delta_{K(X)}}$. Therefore, $Q\in \mathcal{H}^{\delta_{K(X)}}\Leftrightarrow (\cap \mathcal{H}\subseteq Q$ and $Q\in K(X))\Leftrightarrow (Q\in K(X)$ and $Q\cap A\neq\emptyset)$. Now we prove that $\cap \mathcal{H}=A^{\uparrow}$. Clearly, $\cap \mathcal{H}\supseteq A^{\uparrow}$. Conversely, for each $a\in A$ we have $\uparrow\!a\in \diamond A=\mathcal{H}^{\delta_{K(X)}}$, and hence $\cap \mathcal{H}\subseteq \uparrow\!a$. So $\cap \mathcal{H}\subseteq \bigcap\limits_{a\in A}\uparrow\! a=A^{\uparrow}$. Thus $\cap \mathcal{H}=A^{\uparrow}$.

(1) Since $(X, \mathcal{O}(X))$ is a WF space, it follows that $\cap \mathcal{H}\in K(X)$ and thus $\downarrow_{K(X)}\cap \mathcal{H}=\mathcal{H}^{\delta_{K(X)}}=\diamond A=\mbox{cl}_{P^s((X, \mathcal{O}(X)))}\{\uparrow\! a: a\in A\}$. Now we claim that $\cap \mathcal{H}$ is supercompact in $(X, \mathcal{O}(X))$. Suppose $\{U_i\in \mathcal{O}(X):i\in I\}$ with $\cap \mathcal{H}\subseteq \bigcup\limits_{i\in I}U_i$. There exists $I_0\in I^{(<\omega)}\backslash \{\emptyset\}$ such that $\cap \mathcal{H}\subseteq \bigcup\limits_{i\in I_0}U_i$. Thus $\Box(\bigcup\limits_{i\in I_0}U_i)\cap \{\uparrow\! a: a\in A\}\neq\emptyset$. So there is a $i_0\in I_0$ such that $\Box U_{i_0}\cap \{\uparrow\! a: a\in A\}\neq\emptyset$. Hence, $\Box U_{i_0}\cap \downarrow_{K(X)}(\cap \mathcal{H})=\Box U_{i_0}\cap \mbox{cl}_{P^s((X, \mathcal{O}(X)))}\{\uparrow\! a: a\in A\}\neq\emptyset$. Therefore, $\cap\mathcal{H}\subseteq U_{i_0}$, and thus $\cap \mathcal{H}$ is supercompact in $(X, \mathcal{O}(X))$. So there exists $x\in X$ such that $A^{\uparrow}=\cap\mathcal{H}=\uparrow\!x$. Since $Q\in \mathcal{H}^{\delta_{K(X)}}\Leftrightarrow (\cap \mathcal{H}\subseteq Q$ and $Q\in K(X))\Leftrightarrow (Q\in K(X)$ and $Q\cap A\neq\emptyset)$, it follows that $\uparrow\!x\cap A\neq\emptyset$ and thus $x\in A$. Hence $A=\downarrow\!x=\{x\}^{\delta}$. Therefore, $(X, \mathcal{O}(X))$ is quasisober.

(2)As for each $h\in \mathcal{H}$ we have $h\cap A\neq\emptyset$, we can choose a point $a_h$ from $h\cap A$. Thus $\bigcap\limits_{h\in \mathcal{H}}\uparrow\! a_h\subseteq \cap \mathcal{H}=$\\$A^{\uparrow}$. Since $A^{\uparrow}=\bigcap\limits_{a\in A}\uparrow\! a\subseteq \bigcap\limits_{h\in \mathcal{H}}\uparrow\! a_h$ and the specialization order of $(X, \mathcal{O}(X))$ is a sup semilattice , it follows that $\{ \bigvee\limits_{h\in K}a_h:K\in \mathcal{H}^{(<\omega)}\backslash \{\emptyset\}\}$ is a directed subset of $X$ and we have $A^{\uparrow}$$=\bigcap\limits_{a\in A}\uparrow\! a$$=\bigcap\limits_{h\in \mathcal{H}}\uparrow\! a_h$$=\cap\{\uparrow\! \bigvee\limits_{h\in K}a_h:K\in \mathcal{H}^{(<\omega)}\backslash \{\emptyset\}\}$$=\{ \bigvee\limits_{h\in K}a_h:K\in \mathcal{H}^{(<\omega)}\backslash \{\emptyset\}\}^{\uparrow}$ . Thus $A^{\delta}=\{ \bigvee\limits_{h\in K}a_h:K\in \mathcal{H}^{(<\omega)}\backslash \{\emptyset\}\}^{\delta}$. By theorem 3.15 and remark 2.3, we have $A^{\delta}=A$. Thus $A=\{ \bigvee\limits_{h\in K}a_h:K\in \mathcal{H}^{(<\omega)}\backslash \{\emptyset\}\}^{\delta}$. Therefore, $(X, \mathcal{O}(X))$ is quasisober.
\end{proof}

\end{document}